\documentclass[12pt]{amsart}

\usepackage{amsfonts}
\usepackage{amsmath}
\usepackage{amsthm}
\usepackage{url}
\usepackage[all]{xy}
\usepackage{graphicx}
\usepackage{latexsym}
\usepackage{amssymb}
\usepackage[cp850]{inputenc}
\usepackage{epsfig}
\usepackage{psfrag}
\usepackage{amsfonts}

\newtheorem{sat}{Theorem}
\newtheorem{lem}[sat]{Lemma}
			\newtheorem{prop}[sat]{Proposition}
				
\newtheorem*{defi*}{Definition}			\newtheorem*{bei*}{Example}
\newtheorem*{sat*}{Theorem}				\newtheorem*{kor*}{Corollary}
\newtheorem*{rmk*}{Remark}					

\let\ssection=\section
\renewcommand{\section}{\setcounter{equation}{0}\ssection}

\newtheorem*{namedtheorem}{\theoremname}
\newcommand{\theoremname}{testing}
\newenvironment{named}[1]{\renewcommand{\theoremname}{#1}\begin{namedtheorem}}{\end{namedtheorem}}

\theoremstyle{remark}

\newcommand{\BS}{\mathbb S}			
\newcommand{\BF}{\mathbb F}				
				
\newcommand{\BO}{\mathbb O}

\newcommand{\CG}{\mathcal G}

		\newcommand{\CN}{\mathcal N}

\newcommand{\actson}{\curvearrowright}
\newcommand{\D}{\partial}

\DeclareMathOperator{\Out}{Out}		

\DeclareMathOperator{\Map}{Map}

\newcommand{\comment}[1]{}

\DeclareMathOperator{\Aut}{Aut}

\begin{document}

\title[]{Automorphisms of the graph of free splittings}
\author{Javier Aramayona \& Juan Souto}
\thanks{The first author was partially supported by M.E.C. grant MTM2006/14688 and NUI Galway's Millennium Research Fund; the second author was partially supported by NSF grant DMS-0706878 and Alfred P. Sloan Foundation} 
\begin{abstract}
We prove that every simplicial automorphism of the free splitting graph of a free group $\BF_n$ is induced by an outer automorphism of $\BF_n$ for $n\ge 3$.
\end{abstract}
\maketitle


In this note we consider the graph $\CG_n$ of free splittings of the free group $\BF_n$ of rank $n\ge 3$. Loosely speaking, $\CG_n$ is the graph whose vertices are non-trivial free splittings of $\BF_n$ up to conjugacy, and where two vertices are adjacent if they are represented by free splittings admitting a common refinement. The group $\Out(\BF_n)$ of outer automorphisms of $\BF_n$ acts simplicially on $\CG_n$. Denoting by $\Aut(\CG_n)$ the group of simplicial automorphisms of the free splitting graph, we prove: 

\begin{sat}\label{main}
The natural map $\Out(\BF_n)\to\Aut(\CG_n)$ is an isomorphism for $n\ge 3$.
\end{sat}

We briefly sketch the proof of Theorem \ref{main}. We identify $\CG_n$ with the 1-skeleton of the {\em sphere complex} $\BS_n$ and observe that every automorphism of $\CG_n$ extends uniquely to an automorphism of $\BS_n$. It is due to Hatcher \cite{Hatcher} that the sphere complex contains an embedded copy of the {\em spine} $K_n$ of Culler-Vogtmann space. We prove that the latter is invariant under $\Aut(\BS_n)$, and that the restriction homomorphism $\Aut(\BS_n) \to \Aut(K_n)$ is injective. The claim of Theorem \ref{main} then follows from a result of  Bridson-Vogtmann \cite{Bridson-Vogtmann} which asserts that $\Out(\BF_n)$ is the full automorphism group of $K_n$.
\medskip

We are grateful to our motherland for the beauty of its villages.

\section{}
Fixing from now on $n\ge 3$, let $M_n=\#^n\BS^1\times\BS^2$ be the connected sum of $n$ copies of $\BS^1\times\BS^2$. Observe that $\pi_1(M_n)$ is isomorphic to $\BF_n$ and that choosing a basepoint, and conveniently forgetting it afterwards, we can, once and for all, identify 
\begin{equation}\label{eq0}
\pi_1(M_n)\simeq\BF_n
\end{equation}
up to conjugacy. We denote by $\Map(M_n)$ the mapping class group of $M_n$, i.e., the group of isotopy classes of self-diffeomorphisms of $M_n$; observe that \eqref{eq0} induces a homomorphism
\begin{equation}\label{eq1}
\Map(M_n)\to\Out(\BF_n)
\end{equation}
By work of Laudenbach \cite{Laudenbach1} the homomorphism \eqref{eq1} has finite kernel, generated by Dehn-twists along essential embedded 2-dimensional spheres. Recall that an embedded 2-sphere in a 3-manifold is {\em essential} if it does not bound a ball. Two essential embedded 2-spheres $S,S'$ in a 3-manifold are {\em parallel} if they are isotopic. It is due to Laudenbach \cite{Laudenbach1} that $S$ and $S'$ are parallel if and only if they are homotopic to each other. If two parallel essential embedded 2-spheres $S,S'\subset M_n$ are parallel then they bound a submanifold homeomorphic to $\BS^2\times(0,1)$. 

By a {\em system of spheres} in $M_n$ we understand a collection of pairwise disjoint, non-parallel,  essential embedded 2-spheres. A system of spheres is {\em maximal} if it is not properly contained in another system of spheres. Before moving on to more interesting topics we remind the reader of a few useful facts:
\begin{itemize}
\item If $\Sigma\subset M$ is a maximal systems of spheres, then every component of $M\setminus\Sigma$ is homeomorphic to a 3-sphere with 3 balls removed. In particular, $\Sigma$ has $3n-3$ components.
\item If $S$ is a connected component of a maximal system of spheres $\Sigma$, then all components of $M\setminus(\Sigma\setminus S)$ but one are homeomorphic to a 3-sphere with 3 balls removed. The remaining component is either homeomorphic to a 3-sphere with 4 balls removed or to $\BS^1\times\BS^2$ with a ball removed.
\item If $U$ is homeomorphic to a 3-sphere with 4 balls removed, then there are exactly 3 isotopy classes of embedded spheres in $U$ which are neither isotopic to a component of $\D U$ nor bound balls. Namely, every such sphere $S$ separates 2 of the components of $\D U$ from the other two, and the so obtained decomposition of the set of components of $\D U$ determines $S$ up to isotopy.
\item If $U$ is homeomorphic to $\BS^1\times\BS^2$ with a ball removed, then there is single isotopy class of embedded 2-spheres in $U$ which are neither isotopic to a component of $\D U$ nor bound balls.
\item If $S$ is a non-separating component of a maximal system of spheres $\Sigma$, then there is $\Sigma' \subset \Sigma$ with $S\subset \Sigma'$ and $M \setminus \Sigma'$ a 3-sphere with $2n$ balls removed.  
\end{itemize}
The facts listed above follow easily from the existence and uniqueness theorem for prime decompositions of 3-manifolds. See \cite{Hempel} for standard notions of 3-dimensional topology and \cite{Laudenbach1,Laudenbach2} for a treatment of the relation between homotopy and isotopy of embedded 2-spheres in 3-manifolds. 

\section{}
Given an essential embedded 2-sphere $S$ in $M_n$, we denote its isotopy class by $[S]$. 

The {\em sphere complex} $\BS_n$ associated to $M_n$ is the simplicial complex whose vertices are isotopy classes of essential embedded 2-spheres in $M_n$ and where $k+1$ distinct vertices $[S_0],\dots,[S_k]$ span a $k$-simplex if there is a system of spheres $S_0'\cup\dots\cup S_k'$ with $S_i'\in[S_i]$. By definition, the mapping class group of $M_n$ acts simplicially on the sphere complex $\BS_n$. This yields a homomorphism
\begin{equation}\label{eq2}
\Map(M_n)\to\Aut(\BS_n)
\end{equation}
Throughout this note, given a simplicial complex $X$, we denote by $\Aut(X)$ the group of simplicial automorphisms of $X$.

It also follows from work of Laudenbach \cite{Laudenbach1,Laudenbach2} (see also \cite{Hatcher}) that the kernels of the homomorphisms \eqref{eq1} and  \eqref{eq2} are equal. In particular, the action of $\Map(M)$ on $\BS_n$ induces a simplicial action
\begin{equation}
\Out(\BF_n)\actson\BS_n
\label{actionsphere}
\end{equation}
We will observe below that the 1-skeleton $\BS_n^{(1)}$ of $\BS_n$ is equivariantly isomorphic to the free splitting graph $\CG_n$ which we now define.

By a {\em free splitting} of the free group $\BF_n$ we understand an isomorphism between $\BF_n$ and the fundamental group of a graph of groups with trivial edge groups. Two free splittings are said to be {\em equivalent} if there is an $\BF_n$-equivariant isometry between the corresponding Bass-Serre trees. A free splitting of $\BF_n$ is a {\em refinement} of another free splitting if there is a $\BF_n$-equivariant simplicial map from the Bass-Serre tree of the first splitting to the Bass-Serre tree of the second. 

In the sequel, we will pass freely between free splittings, the associated graph of group decompositions and the associated Bass-Serre trees. Similarly, we will say for instance that a Bass-Serre tree is a refinement of some other Bass-Serre tree. We hope that this will cause no confusion.

The {\em free splitting graph} $\CG_n$ of $\BF_n$ is the simplicial graph whose vertices are equivalence classes of free splittings of $\BF_n$ whose corresponding graph of groups have a single edge. Two vertices of $\CG_n$ are adjacent if they are represented by free splittings which have a common refinement. Observe that $\Out(\BF_n)$ acts on $\CG_n$ by simplicial automorphisms.

\begin{lem}\label{lem:split-sphere}
There is a simplicial isomorphism $\BS_n^{(1)}\to\CG_n$ conjugating the standard actions of $\Out(\BF_n)$.
\end{lem}

Lemma \ref{lem:split-sphere} is surely known to all experts in the field; we sketch a proof for completeness.

\begin{proof}[Sketch of proof]
To an essential embedded sphere $S$ in $M_n$ we associate its dual tree, that is, the Bass-Serre tree of the graph of groups decomposition of $\pi_1(M_n)\simeq\BF_n$ given by the Seifert-van Kampen theorem.  Isotopic spheres yield equivalent free splittings and hence we obtain a vertex of $\CG_n$ for every vertex of $\BS_n$. The dual tree to the union of two disjoint embedded spheres $S,S'$ is the Bass-Serre tree of a free splitting of $\BF_n\simeq\pi_1(M_n)$ which is clearly a refinement of the Bass-Serre trees associated to $S$ and $S'$. In other words, the map between vertices extends to a map $\BS_n^{(1)}\to\CG_n$.

Given now a free splitting with a single edge, consider the associated Bass-Serre tree  $T$, and let $\phi:\tilde M\to T$ be an equivariant simplicial map; here $\tilde M$ is the universal cover of $M$ endowed with a simplicial structure lifted from $M$ and we have identified $\pi_1(M)$ with $\BF_n$ as per \eqref{eq0}. Given a regular value $\theta$ on an edge of $T$ we consider its preimage $\phi^{-1}(\theta)$. A well-known surgery argument using Dehn's lemma implies that, up to $\BF_n$-equivariant homotopy of $\phi$, we may assume that  every component of the projection of $\phi^{-1}(\theta)$ to $M$ is $\pi_1$-injective, does not bound a ball, and that no two components of are parallel. Since the edge stabilizers are trivial, we deduce that every component of the projection of $\phi^{-1}(\theta)$ to $M$ is a sphere. Let $T'$ be the dual tree associated to the projection of $\phi^{-1}(\theta)$ to $M$. By construction the map $\phi:\tilde M\to T$ induces an injective simplicial map $T'\to T$. Since the tree $T$ is minimal, this map is surjective as well; hence, the projection of $\phi^{-1}(\theta)$ to $M$ is a single sphere. In this way, we have associated a vertex in $\BS_n$ to every vertex in $\CG_n$. An analogous argument applies to the edges.
\end{proof}

In light of Lemma \ref{lem:split-sphere}, the claim of Theorem \ref{main} will follow once we prove that $\Out(\BF_n)$ is the full automorphism group of $\BS_n^{(1)}$. The first step in this direction is to observe that every automorphism of $\BS_n^{(1)}$ is induced by an automorphism of the whole sphere complex. 

\begin{lem}\label{lem:flag}
$\BS_n$ is a flag complex; in particular, every automorphism of $\BS_n^{(1)}$ is the restriction of a unique automorphism of $\BS_n$.
\end{lem}

Recall that a simplicial complex is {\em flag} if every complete subgraph on $r+1$ vertices contained in the $1$-skeleton is the 1-skeleton of an $r$-simplex.

Lemma \ref{lem:flag} has also been established in \cite[Theorem 3.3]{gadgil}. Again, we prove it only for completeness.

\begin{proof}
By the very definition of simplicial complex as a subset of the power set of the set of vertices, a simplex in a simplicial complex is uniquely determined by its set of vertices. Hence, it is clear that every automorphism of the 1-skeleton of a flag simplicial complex extends uniquely to an automorphism of the whole complex. Therefore, it suffices to prove the first claim of the lemma.

We will argue by induction. For complete graphs with two vertices, there is nothing to prove. So, suppose that the claim has been proved for all complete graphs with $k-1\ge 2$ vertices and let $v_1,\dots,v_k$ be vertices in $\BS_n^{(1)}$ spanning a complete graph. Applying the induction assumption three times, namely to the complete subgraphs spanned by $\{v_1,\dots,v_{k-1}\}$, $\{v_1,\dots,v_{k-2},v_k\}$ and $\{v_1,\dots,v_{k-3},v_{k-1},v_k\}$ and isotopying spheres to avoid redundancies, we can find spheres 
$$S_1,\dots,S_{k-1},S_k,S_k'\subset M_n$$ 
satisfying:
\begin{enumerate}
\item $S_1,\dots,S_{k-1}$ represent $v_1,\dots,v_{k-1}$ respectively, and both $S_k$ and $S_k'$ represent $v_k$.
\item $S_i\cap S_j=\emptyset$ for $i,j=1,\dots,k-1$ with $i\neq j$.
\item $S_i \cap S_k=\emptyset$ for $i=1,\dots,k-2$, and
\item $S_i \cap S_k'=\emptyset$ for $i=1,\dots,k-3$ and $i=k-1$.
\end{enumerate}
Choose a maximal system of spheres $\Sigma$ with
$$S_1,\dots,S_{k-1}\subset\Sigma$$
By \cite[Proposition 1.1]{Hatcher} we can assume that the sphere $S_k$ is in {\em normal form} with respect to $\Sigma$. Recall that this means that $S_k$ is either contained in $\Sigma$ or meets $\Sigma$ transversally and that, in the latter case, the closure $P$ of any component of $S_k\setminus\Sigma$ satisfies:
\begin{itemize}
\item $P$ meets any component of $\Sigma$ in at most one circle.
\item $P$ is not a disk isotopic, relative to its boundary, to a disk in $\Sigma$.
\end{itemize}
Similarly, we assume that $S_k'$ is also in normal form with respect to $\Sigma$.

By assumption, the spheres $S_k$ and $S_k'$ are isotopic. By \cite[Proposition 1.2]{Hatcher}, there is a homotopy $(S(t))_{t\in[0,1]}$ by immersed spheres with $S(0)=S_k$ and $S(1)=S_k'$ satisfying:
\begin{itemize}
\item If $S(0)\subset\Sigma$ then $S(t)\subset\Sigma$ for all $t$ and hence $S(1)=S(0)$.
\item If $S(0)\not\subset\Sigma$ then $S(t)$ is transverse to $\Sigma$ for all $t$.
\end{itemize}
If we are in the first case then $S_1,\dots,S_k\subset\Sigma$ are pairwise disjoint and we are done. Otherwise, recall that $S(1)=S_k'$ does not intersect $S_{k-1}$. Since transversality is preserved through the homotopy, it follows that $S(0)=S_k$ does not intersect $S_{k-1}$ either. We deduce from (1) and (2) above that the spheres $S_1,\dots,S_k$ are pairwise disjoint. This concludes the proof of the induction step, thus showing that $\BS_n$ is a flag complex.
\end{proof}

\section{}
We now briefly recall the definition of Culler-Vogtmann space $CV_n$, also called {\em outer space}. A point in $CV_n$ is an equivalence class of marked metric graphs $X$ with $\pi_1(X)=\BF_n$, of total length $1$, without vertices of valence one and without separating edges. Two such marked graphs are {\em equivalent} if they are isometric via an isometry in the correct homotopy class. Two marked metric graphs $X,Y$ are close in $CV_n$ if, for some $L$ close to $1$, there are $L$-Lipschitz maps $X\to Y$ and $Y\to X$ in the correct homotopy classes. See \cite{Culler-Vogtmann}  for details.

Following Hatcher \cite{Hatcher} we denote by $\BO_n$ the similarly defined space where one allows the graphs to have separating edges (but no vertices of valence one). Observe that
$$CV_n\subset\BO_n$$
The authors are tempted to refer to $\BO_n$ as {\em hairy outer space}. 

The group $\Out(\BF_n)$ acts on $\BO_n$ by changing the marking. This action preserves $CV_n$ as a subset of $\BO_n$. The interest of the space $\BO_n$ in our setting is that, by work of Hatcher, it is $\Out(\BF_n)$-equivariantly homeomorphic to a subset of the sphere complex $\BS_n$. We now describe this homeomorphism.

From now on we interpret points in the sphere complex $\BS_n$ as weighted sphere systems in $M$; to avoid redundancies we assume without further mention that all weights are positive. As in \cite{Hatcher}, let $\BS_n^\infty$ be the subcomplex of $\BS_n$ consisting of those elements $\sum_ia_iS_i\in\BS_n$ such that $M\setminus\bigcup_iS_i$ has at least one non-simply connected component. 

To a point $\sum_ia_iS_i\in\BS_n\setminus\BS_n^\infty$ we associate the dual graph to $\bigcup_iS_i$ and declare the edge corresponding to $S_i$ to have length $a_i$. This yields a map
$$\BS_n\setminus\BS_n^\infty\to\BO_n$$
Hatcher \cite[Appendix]{Hatcher} shows:

\begin{prop}[Hatcher]
The map $\BS_n\setminus\BS_n^\infty\to\BO_n$ is an $\Out(\BF_n)$-equivariant  homeomorphism.
\end{prop}

Besides introducing $CV_n$, in \cite{Culler-Vogtmann} Culler and Vogtmann define what is called the {\em spine} $K_n$ of $CV_n$. Considering $CV_n$ as a subset of $\BS_n$, the spine $K_n$ is the maximal full subcomplex of the first barycentric subdivision of $\BS_n$ which is contained in $CV_n$ and is disjoint from $\BS_n^\infty$. 

By construction, \eqref{actionsphere} induces an action $\Out(\BF_n) \actson K_n$ by simplicial automorphisms and hence a homomorphism 
\begin{equation}
\Out(\BF_n) \to \Aut(K_n).
\label{eq:bv}
\end{equation} 
The key ingredient in the proof of Theorem \ref{main} in the next section is the following result of Bridson and Vogtmann \cite{Bridson-Vogtmann},

\begin{sat}[Bridson-Vogtmann]
For $n\geq 3$, the homomorphism \eqref{eq:bv} is an isomorphism. 
\end{sat}

Observe that, for $n=2$, $\Aut(K_2)$ is uncountable.

\section{}
In this section we prove Theorem \ref{main}, whose statement we now recall:

\begin{named}{Theorem \ref{main}}
The natural map $\Out(\BF_n)\to\Aut(\CG_n)$ is an isomorphism for $n\ge 3$.
\end{named}
To begin with, we remind the reader that by Lemma \ref{lem:split-sphere} the free splitting graph $\CG_n$ is $\Out(\BF_n)$-equivariantly isomorphic to the 1-skeleton $\BS_n^{(1)}$ of the sphere complex $\BS_n$. By Lemma \ref{lem:flag}, $\BS_n$ is flag and hence the claim of Theorem \ref{main} will follow once we prove that the simplicial action $\Out(F_n) \actson \BS_n$ in \eqref{actionsphere} induces an isomorphism 
\begin{equation}
\Out(\BF_n) \to \Aut(\BS_n).
\label{manolo}
\end{equation}
We start by proving that $\BS_n^\infty$ is invariant under $\Aut(\BS_n)$.

\begin{lem}\label{preserve}
Every automorphism of $\BS_n$ preserves the subcomplex $\BS_n^\infty$.
\end{lem}
\begin{proof}
We first observe that every simplex in $\BS_n^\infty$ is contained in a codimension 1 simplex which is also contained in $\BS_n^\infty$. To see this, let $v_0,\dots,v_k$ be vertices of $\BS^n$ spanning a $k$-simplex in $\BS_n^\infty$. We represent these vertices by pairwise disjoint embedded spheres $S_0,\dots,S_k$. By the definition of $\BS_n^\infty$ there is a component $U$ of $M\setminus\bigcup_i S_i$ containing a sphere $S$ with $U\setminus S$ connected. Let $\alpha$ be an embedded curve in $U$ intersecting $S$ exactly once, $V$ a closed regular neighborhood of $S\cup\alpha$ and $S'$ the boundary of $V$; $S'$ is an essential embedded sphere. Cutting $V$ open along $S$ we obtain a 3-sphere with 3 balls removed. In particular, every embedded sphere in $V$ disjoint from $S$ is parallel to one of $S$ or $S'$. Let $\Sigma$ be a maximal sphere system containing $S_0,\dots,S_k,S,S'$. The simplex in $\BS_n$ determined by $\Sigma\setminus S$ has codimension 1 and is contained in $\BS_n^\infty$.

The upshot of this observation is that the claim of Lemma \ref{preserve} follows once we show that codimension 1 simplices contained in $\BS_n^\infty$ can be characterized in terms of simplicial data. Namely, we claim that a codimension 1 simplex is contained in $\BS_n^\infty$ if and only if it is contained in a unique top-dimensional simplex. Suppose that a system of spheres $\Sigma$ determines a codimension 1 simplex $[\Sigma]$ and consider $M\setminus\Sigma$. As mentioned above, all components of $M\setminus\Sigma$ but one are homeomorphic to a 3-sphere with 3 balls removed. The remaining component, call it $U$, is either a 3-sphere with 4 balls removed or $\BS^1\times\BS^2$ with a ball removed. In the first case, $[\Sigma]$ is contained in $\BS_n\setminus\BS_n^\infty$ and, since $U$ contains 3 different essential spheres, the simplex $[\Sigma]$ is a face of 3 distinct maximal simplices. On the other hand, if $U$ is homeomorphic to $\BS^1\times\BS^2$ with a ball removed, then $[\Sigma]\subset\BS_n^\infty$ and $U$ contains a unique embedded sphere which does not bound a ball and is not parallel to $\D U$. We deduce that $[\Sigma]$ is a face of a unique maximal simplex. 

This concludes the proof of Lemma \ref{preserve}.
\end{proof}

It follows from Lemma \ref{preserve} that $\Aut(\BS_n)$ preserves the hairy Culler-Vogtmann space $\BO_n=\BS_n\setminus\BS_n^\infty$. We now prove that $\Aut(\BS_n)$ preserves the spine $K_n$ of $CV_n\subset\BO_n$ as well.

\begin{lem}\label{preserve-spine}
Every simplicial automorphism of $\BO_n$ preserves the spine $K_n$ of $CV_n$.
\end{lem}

Abusing terminology, we will say from now on the a simplex $\sigma$ is contained in, for instance, $CV_n$ if the associated open simplex is.

\begin{proof}
Since the spine of $CV_n$ is defined simplicially, it suffices to show that every automorphism of $\BO_n$ preserves $CV_n$ itself. As in the proof of Lemma \ref{preserve} it suffices to characterize combinatorially the maximal simplices (whose interior is) contained in $CV_n$; equivalently, we characterize those in $\BO_n\setminus CV_n$.

We claim that a top dimensional simplex $\sigma$ is contained in $\BO_n\setminus CV_n$ if and only if the following is satisfied:
\begin{itemize}
\item[(*)] $\sigma$ has a codimension one face $\tau$ such that whenever $\eta\subset\sigma$ is a face contained in $\BS_n\setminus\BS_n^\infty$ then $\eta\cap\tau\subset\BS_n\setminus\BS_n^\infty$.
\end{itemize}
We first prove that every top-dimensional simplex $\sigma\subset\BO_n\setminus CV_n$ satisfies (*). Let $\sigma$ be represented by a maximal sphere system $\Sigma$. The assumption that $\sigma\subset\BO_n\setminus CV_n$ implies that $\Sigma$ has a component $S$ which separates $M$. Let $\tau$ be the codimension 1 face of $\sigma$ determined by $\Sigma\setminus S$ and suppose that $\eta\subset\sigma$ is a face contained in $\BS_n\setminus\BS_n^\infty$. Denote by $\Sigma'$ the subsystem of $\Sigma$ corresponding to $\eta$. Since $S$ is separating and all components of $M\setminus\Sigma'$ are simply connected, it follows from the Seifert-van Kampen theorem that every component of $M\setminus(\Sigma'\setminus S)$ is simply connected as well. In other words, the simplex $\eta\cap\tau$ is contained in $\BS_n\setminus\BS_n^\infty$ as claimed.

Suppose now that the top dimensional simplex $\sigma$ is contained in $CV_n$; we are going to prove that (*) is not satisfied for any codimension 1 face $\tau\subset\sigma$. Continuing with the same notation, let $S_1\subset\Sigma$ be the sphere such that $\Sigma\setminus S_1$ represents $\tau$. Since $S_1$ is by assumption non-separating we can find other $n-1$ components $S_2,\dots,S_n$ of $\Sigma$ such that $M\setminus\bigcup_i S_i$ is homeomorphic to a 3-sphere with $2n$ balls removed. The simplex $\eta$ associated to the system $S_1\cup\dots\cup S_n$ is contained in $\BS_n\setminus\BS_n^\infty$. On the other hand, the complement of $S_2\cup\dots\cup S_n$ is not simply connected. Hence, the simplex $\eta\cap\tau$ is not contained in $\BS_n\setminus\BS_n^\infty$. This proves that (*) is not satisfied for the face $\tau$. As $\tau$ is arbitrary, this concludes the proof of Lemma \ref{preserve-spine}.
\end{proof}

It follows from Lemma \ref{preserve-spine} and the Bridson-Vogtmann theorem that there is a homomorphism
\begin{equation}\label{eq:bla}
\Aut(\BS_n)\to\Aut(K_n)\simeq\Out(\BF_n)
\end{equation}
As mentioned above, by Lemma \ref{lem:split-sphere} and Lemma \ref{lem:flag}, we have identified $\Aut(\CG_n)$ with $\Aut(\BS_n)$. In particular, the proof of Theorem \ref{main} boils down to showing that the homomorphism \eqref{eq:bla} is injective. This is the content of the following lemma:

\begin{lem}\label{lem:final}
The identity is the only automorphism of $\BS_n$ acting trivially on the spine $K_n$.
\end{lem}

\begin{proof}
Recall that $K_n$ is the maximal full subcomplex of the first barycentric subdivision of $\BS_n$ which is contained in  $CV_n$ and is disjoint from $\BS_n^\infty$. The interior of every simplex contained in $CV_n$ intersects $K_n$. In particular, an automorphism $\alpha$ of $\BS_n$ which acts trivially on the spine $K_n$ maps every simplex in $CV_n$ to itself. We claim that the restriction of $\alpha$ to $CV_n$ is actually the identity. 

Let $\Sigma$ be a sphere system in $M$ determining a top dimensional simplex $\sigma$ in $CV_n$ and let $S$ be a component of $\Sigma$. We claim that the codimension 1 face given by $\Sigma\setminus S$ is also contained in $CV_n$; in order to see this it suffices to prove that it is contained in $\BS_n\setminus\BS_n^\infty$. If this were not the case, then the unique component $U$ of $M\setminus(\Sigma\setminus S)$ distinct from a 3-sphere with 3 balls removed is homeomorphic to $\BS^1\times\BS^2$ with a ball removed. The boundary of $U$ is a connected component of $\Sigma$ which separates $M$; a contradiction to the assumption that $\sigma\subset CV_n$. It follows that the automorphism $\alpha$ maps the codimension 1 face of $\sigma$ determined by $\Sigma\setminus S$ to itself. In particular, $\alpha$ has to fix the opposite vertex $[S]$ of $\sigma$; $[S]$ being arbitrary, we have proved that $\alpha$ is the identity on $\sigma$. Hence, $\alpha$ is the identity on $CV_n$.

We are now ready to prove that $\alpha$ fixes every vertex $[S]$ in $\BS_n$; once we have done this, the claim of the lemma will follow. Given a vertex $[S]$ there are two possibilities. If the sphere $S$ is non-separating we can extend $S$ to a maximal sphere system $\Sigma$ with no separating components. The simplex determined by $\Sigma$ is contained in $CV_n$ and is hence fixed by $\alpha$. If $S$ is separating let $U$ and $V$ be the two components of $M\setminus S$. For a suitable choice of $r$, we identify $U$ with the complement of a ball $B$ in the connected sum $\#^r\BS^1\times\BS^2$ of $r$ copies of $\BS^1\times\BS^2$. Similarly, $V$ is the complement of a ball in the connected sum of $s$ copies of $\BS^1\times\BS^2$. The cases $r=1$ and $s=1$ are minimally special; they are left to the reader. 

We choose a maximal sphere system $\Sigma_U$ in $\#^r\BS^1\times\BS^2$ whose dual graph has no separating edges. Choosing some sphere $S'$ in $\Sigma_U$ we take a small regular neighborhood $\CN(S')$ of $S'$ in $\#^r\BS^1\times\BS^2$. Up to isotopy we may assume that the ball $B$ is contained in $\CN(S')$. The collection of spheres $\D\CN(S')\cup\Sigma_U\setminus S'$ is contained in $U=\#^r\BS^1\times\BS^2\setminus B$ and hence determines a sphere system $\Sigma_U'$ in $M$. We do a similar construction for $V$ obtaining a system $\Sigma_V'$ and set
$$\Sigma=S\cup\Sigma_U'\cup\Sigma_V'$$
The simplex determined by $\Sigma\setminus S$ is contained in $CV_n$ and is thus fixed by the automorphism $\alpha$. Also, it follows from the construction that $S$ is the unique separating sphere contained in $M\setminus(\Sigma\setminus S)$. By the above, all the vertices determined by any other sphere disjoint from $\Sigma\setminus S$ are fixed by $\alpha$. It follows that $\alpha$ fixes the vertex $[S]$ as claimed.
\end{proof}

As mentioned above, the proof of Lemma \ref{lem:final} concludes the proof of Theorem \ref{main}.\qed

\bigskip

\noindent Department of Mathematics, National University of Ireland, Galway. \newline \noindent
\texttt{javier.aramayona@nuigalway.ie}

\bigskip

\noindent Department of Mathematics, University of Michigan, Ann Arbor. \newline \noindent
\texttt{jsouto@umich.edu}


\begin{thebibliography}{99}

\bibitem{Bridson-Vogtmann}
M. Bridson and K. Vogtmann, {\em The symmetries of outer space}, Duke Math.J. 106 (2001). 

\bibitem{Culler-Vogtmann}
M. Culler and K. Vogtmann, {\em Moduli of graphs and automorphisms of free groups}, Invent. 
Math. 84 (1986).

\bibitem{gadgil} 
S. Gadgil and S. Pandit, {\em Algebraic and geometric intersection numbers for free groups}. Prepint, {\url {arXiv:0809.3109}}.

\bibitem{Hatcher}
A. Hatcher, {\em Homological stability for automorphism groups of free groups}, Comment. Math. Helvetici 70 (1995).

\bibitem{Hempel}
Hempel, {\em 3-manifolds}, Princeton University Press, 1976.

\bibitem{Laudenbach1}
F. Laudenbach, {\em Sur les 2-spheres d'une vari\'et\'e de dimension 3}, Annals of Math. 97 (1973). 

\bibitem{Laudenbach2}
F. Laudenbach, {\em Topologie de la dimension trois: homotopie et isotopie}, Ast\'erisque 12, 1974.


\end{thebibliography}
\end{document}